\documentclass[11pt,twoside]{amsart}
 \usepackage[T1]{fontenc}
\usepackage[utf8]{inputenc}
\usepackage[usenames]{color}
\usepackage{graphicx}
\usepackage{float}
\usepackage{fancyhdr}
\usepackage{geometry}
\usepackage{subcaption} 
\input{xy}
\xyoption{all}
\xyoption{poly}
\usepackage[all]{xy}

 \usepackage{amsmath,amsthm,amsfonts,amssymb}
      \theoremstyle{plain}
      \newtheorem{theorem}{Theorem}[section]
      \newtheorem*{theorem*}{Theorem}
      \newtheorem{lemma}[theorem]{Lemma}
      \newtheorem{corollary}[theorem]{Corollary}
      \newtheorem{proposition}[theorem]{Proposition}
      \newtheorem{conjecture}[theorem]{Conjecture}
      \newtheorem*{aproposition}{Proposition \ref{propWebbConjReformulated}}
      \newtheorem*{bproposition}{Proposition \ref{propMxlRank}}
      \newtheorem*{atheorem}{Theorem \ref{theoremApGGcontractil}}
      \newtheorem*{aconjecture}{Conjecture \ref{strongerConjecture}}
      \theoremstyle{definition}
	  \newtheorem{example}[theorem]{Example}
      \newtheorem{definition}[theorem]{Definition}
      
      \theoremstyle{remark}
      \newtheorem{remark}[theorem]{Remark}

 \newcommand\ZZ{{\mathbb{Z}}}

 \renewcommand\SS{{\mathbb{S}}}
 \renewcommand\AA{{\mathbb{A}}}

 \def\A{{\mathcal A}}
 \def\B{{\mathcal B}}

 \def\F{{\mathcal F}}

 \def\K{{\mathcal K}}

 \def\N{{\mathcal N}}
 \def\O{{\mathcal O}}
 \def\P{{\mathcal P}}
 
 \def\R{{\mathcal R}}
 \def\S{{\mathcal S}}

 \def\X{{\mathcal X}}

 \def\PSL{{\text{PSL}}}
 \def\Syl{{\text{Syl}}}
 
 \def\Max{{\text{Max}}}

\newcommand{\weak}{\underset{w}{\approx}}
\newcommand{\normal}{\trianglelefteq}

\newcommand\gen[1]{\left\langle#1\right\rangle}

\usepackage{tikz}

      \makeatletter
      \def\@setcopyright{}
      \def\serieslogo@{}
      \makeatother
   
\begin{document}

\title [A stronger reformulation of Webb's conjecture]{A stronger reformulation of Webb's conjecture in terms of finite topological spaces}
   \author{Kevin Iv\'an Piterman}
   \address{Departamento  de Matem\'atica \\IMAS-CONICET\\
 FCEyN, Universidad de Buenos Aires. Buenos Aires, Argentina.}
\email{kpiterman@dm.uba.ar}

\thanks{Partially supported by grants UBACyT 20020160100081BA}

   \begin{abstract}
We investigate a stronger formulation of Webb's conjecture on the contractibilty of the orbit space of the $p$-subgroup complexes in terms of finite topological spaces. The original conjecture, which was first proved by Symonds and, more recently, by Bux, Libman and Linckelmann, can be restated in terms of the topology of certain finite spaces. We propose a stronger conjecture, and prove various particular cases by combining fusion theory of finite groups and homotopy theory of finite spaces.
   \end{abstract}



\maketitle

\section{Introduction}

Let $G$ be a finite group and let $p$ be a prime number dividing $|G|$, the order of $G$. Denote by $\S_p(G)$ the poset of non-trivial $p$-subgroups of $G$ and by $\A_p(G)$ the subposet of non-trivial elementary abelian $p$-subgroups of $G$. The study of the posets of $p$-subgroups began in the seventies with K. Brown's proof of the Homological Sylow theorem, which states that $\chi(\K(\S_p(G)))\equiv 1\mod |G|_p$ (see \cite{Bro75}). Here $|G|_p$ is the greatest power of $p$ dividing $|G|$ and $\K(X)$ is the classifying space of the finite poset $X$, i.e. the simplicial complex whose simplices are the non-empty chains of $X$. Some years later, D. Quillen introduced the subposet $\A_p(G)$ and showed that the inclusion $\A_p(G)\hookrightarrow \S_p(G)$ induces a homotopy equivalence $\K(\A_p(G))\hookrightarrow\K(\S_p(G))$ at the level of complexes \cite{Qui78}. Quillen noticed that $\K(\S_p(G))$ is contractible if $G$ has a non-trivial normal $p$-subgroup and conjectured the converse. This problem remains open but there have been significant advances in this direction, especially with the work of M. Aschbacher and S.D. Smith \cite{AS93}. The standard way for studying the relationship between the $p$-subgroup posets and the algebraic properties of $G$ is by means of the topological properties of their associated simplicial complexes (see for example \cite{Asc93,AK90,AS93,Bou84,Dia16, HI88, Kso03,Kso04,Smi11,TW91,Web87}). In \cite{Web87} P. Webb related these posets with representation theory and conjectured that the orbit space of the topological space $\K(\S_p(G))$ is contractible. In \cite{Sym98} P. Symonds proved Webb's conjecture by using Whitehead's theorem: he showed that it is simply connected and acyclic. Another proof of Webb's conjecture is due to K.U. Bux (see \cite{Bux99}) via Bestvina-Brady's approach to Morse Theory.  More recent versions and generalizations of this conjecture (which use orbit spaces arising from fusion systems) can be found in \cite{Lib08,Lin09}. However, all these works rely on the study of the homotopy properties of certain CW-complexes associated to the posets.

In our previous work \cite{MP18} we have investigated the posets $\A_p(G)$ and $\S_p(G)$ from the point of view of the homotopy theory of finite topological spaces, following R.E. Stong's approach of \cite{Sto84}. Stong showed that $\A_p(G)$ and $\S_p(G)$ have, in general, different homotopy types viewed as finite topological spaces (see also \cite{MP18}), although their associated simplicial complexes $\K(\A_p(G))$ and $\K(\S_p(G))$ are always $G$-homotopy equivalent (see \cite{Qui78,TW91}). In this context, Quillen's conjecture is equivalent to saying that if $\S_p(G)$ is a homotopically trivial finite space then it is contractible as finite space. Recall that a topological space is homotopically trivial if all its homotopy groups are trivial. It is known that, in the context of finite spaces, being homotopically trivial is in general strictly weaker than being contractible (see \cite{Bar11a} for  examples of homotopically trivial but not contractible finite spaces). In \cite{MP18} it is shown that Quillen's conjecture cannot be restated in terms of the finite space $\A_p(G)$. Concretely, in \cite{MP18} we answered a question raised by Stong in \cite{Sto84} by exhibiting a poset $\A_p(G)$ which is homotopically trivial but not contractible as finite space. 

In this paper, we continue the study of the posets of $p$-subgroups from the point of view of finite spaces. In this direction, Webb's conjecture can be reformulated by saying that $\A_p(G)'/G$ and $\S_p(G)'/G$ are homotopically trivial finite spaces. Here $X'$ denotes the subdivision of the poset $X$ and it is the poset of non-empty chains of $X$. In fact, in all the examples that we have explored $\A_p(G)'/G$ turned out to be contractible, and we believe that a stronger statement of Webb's conjecture is true: $\A_p(G)'/G$ is a contractible finite space. We will prove this fact in some particular cases.

The paper is organized as follows. In Section 2, we recall some basic facts on finite spaces and equivariant theory for simplicial complexes. In Section 3 we restate Webb's original conjecture in terms of finite spaces:

\begin{aproposition}
(Webb's Conjecture) If $K\subseteq \K(\S_p(G))$ is a $G$-invariant subcomplex which is $G$-homotopy equivalent to $\K(\S_p(G))$, the finite space $\X(K)/G$ is homotopically trivial. In particular, this holds for $K \in\{\K(\S_p(G)), \K(\A_p(G)), \K(\B_p(G)), \R_p(G)\}$.
\end{aproposition}

Here $\X(K)$ denotes the face poset of the simplicial complex $K$. Note that $X' = \X(\K(X))$ for any poset $X$. The subposet $\B_p(G)\subseteq\S_p(G)$ consists of the $p$-radical subgroups of $G$ and $\R_p(G)\subseteq \K(\S_p(G))$ is the subcomplex with simplices $(P_0 < P_1 < \ldots < P_r)$ where $P_i\normal P_r$ for all $i$. Both $\K(\B_p(G))$ and $\R_p(G)$ have the same $G$-homotopy type as $\K(\S_p(G))$ (see \cite{Bou84,Smi11,TW91}).

We exhibit examples of non-contractible posets $\S_p(G)'/G$ and $\B_p(G)'/G$. For $\A_p(G)'/G$ we propose the following conjecture.

\begin{aconjecture}
The finite space $\A_p(G)'/G$ is contractible.
\end{aconjecture}

Section 4 is devoted to prove that $\A_p(G)/G$ (without subdividing) is a contractible finite space. The finite space $\S_p(G)/G$ is always contractible because it has a maximum (the class of any Sylow $p$-subgroup). However $\A_p(G)/G$ may have no maximum and the conclusion is not immediate. 

\begin{atheorem}
The finite poset $\A_p(G)/G$ is conically contractible.
\end{atheorem}

The techniques used in the proof of this theorem involve the study of the behaviour of the fully centralized elementary abelian $p$-subgroups inside a fixed Sylow $p$-subgroup of $G$.

In Section 5 we show some particular cases for which the finite spaces $\A_p(G)'/G$, $\S_p(G)'/G$ and $\B_p(G)'/G$ are contractible.

\begin{bproposition}
If $\A_p(G)\subseteq \S_p(G)$ is a strong deformation retract, both $\A_p(G)'/G$ and $\S_p(G)'/G$ are contractible finite spaces. In particular, this holds when the Sylow $p$-subgroups are abelian. Moreover, in the latter case $\B_p(G)\subseteq \S_p(G)$ is a strong deformation retract and $\B_p(G)'/G$ is a contractible finite space.
\end{bproposition}

In the following theorem we summarize the cases for which we prove that $\A_p(G)'/G$ is contractible.

Recall that $r_p(G)$, the $p$-rank of $G$, is the maximum integer $r$ such that $G$ has an elementary abelian $p$-subgroup of order $p^r$. Denote by $\Syl_p(G)$ the set of Sylow $p$-subgroups of $G$. Fix $P\in \Syl_p(G)$ and let $\Omega$ be the subgroup of $P$ of central elements in $P$ of order dividing $p$.

\begin{theorem*}
$\A_p(G)'/G$ is a contractible finite space in the following cases:
\begin{enumerate}
\item $|G| = p^\alpha.q$ with $q$ a prime number,
\item $\A_p(G)$ is contractible (as finite space),
\item $r_p(G) - r_p(\Omega) \leq 1$,
\item $r_p(G) - r_p(\Omega) = 2$ and $r_p(G)\geq \log_p(|G|_p) - 1$,
\item $|G|_p\leq p^4$.
\end{enumerate}
\end{theorem*}

We work with conjugation classes of chains of $p$-subgroups inside the fixed Sylow $p$-subgroup $P$. Therefore, almost all results can be carried out in a general saturated fusion system over $P$. Moreover, our techniques rely on analyzing how we can control the fusion in certain chains of elementary abelian $p$-subgroups of $P$.

\section{Preliminaries on finite spaces and $G$-complexes}

We recall first some basic facts on the homotopy theory of finite topological spaces. For more details, we refer the reader to \cite{Bar11a, McC66, Sto66}.

In what follows, all posets and simplicial complexes considered are assumed to be finite.

For a finite poset $X$ we can study its topological properties by means of its associated simplicial complex $\K(X)$, whose simplices are the non-empty chains of $X$. However, there is an intrinsic topology on $X$ whose open sets are given by the downsets, i.e. the subsets $U$ of $X$ such that $y\leq x$ and $x\in U$ implies $y\in U$. The minimal open sets are the sets $U_x = \{y\in X:y\leq x\}$ for $x\in X$, and they form a basis for this topology. With this topology, $X$ becomes a finite $T_0$-space. Conversely, any finite $T_0$-space $X$ has a natural poset structure by setting $x \leq y$ if the minimum open set containing $x$ is contained in the minimum open set containing $y$. A map between posets is a continuous map if and only if it is an order-preserving map. Moreover, two continuous maps $f,g:X\to Y$ are homotopic (in the classical sense) if and only if there exist continuous maps $f_0,\ldots,f_n:X\to Y$ such that $f_0 = f$, $f_n = g$ and for each $0\leq i < n$, $f_i \leq f_{i+1}$, i.e. $f_i(x)\leq f_{i+1}(x)$ for every $x\in X$, or $f_i \geq f_{i+1}$.

The relation between the topology of $X$ and that of $\K(X)$ is given by McCord's theorem (see \cite{McC66}): there exists a natural weak equivalence $\mu_X:\K(X)\to X$ defined by
\[ \mu_X\left(\sum_{i=0}^n t_i x_i\right) = \max\{x_0,\ldots,x_n\}.\]
Recall that a continuous map between topological spaces is a weak equivalence if it induces isomorphisms between all the homotopy groups, and hence between the homology groups. In particular, $\K(X)$ and $X$ have the same homotopy groups and homology groups. By naturality of the McCord's map, two finite posets are weak equivalent if and only if their associated simplicial complexes are homotopy equivalent. In general, $\K(X)$ and $X$ do not have the same homotopy type, and weak equivalences between finite posets may not be homotopy equivalences. For example, $\A_p(G)\hookrightarrow \S_p(G)$ is a weak equivalence (see \cite{Qui78}) but in general it is not always a homotopy equivalence since they do not have the same homotopy type (see \cite{MP18,Sto84}). In conclusion, the classical theorem of J.H.C. Whitehead is no longer true in the context of finite spaces, and the notion of homotopy equivalence in the context of finite spaces is strictly stronger than the corresponding notion in the context of simplicial complexes. See \cite{Bar11a,MP18} for non-contractible homotopically trivial posets.

The classification of homotopy types of finite spaces can be done combinatorially. This was studied by Stong in a previous article \cite{Sto66}, using  the notion of \it beat point. \rm  An element $x\in X$ is called a {\it down beat point} if $\hat{U}_x=\{ y \in X: y<x\}$ has a maximum, and it is an {\it up beat point} if $\hat{F}_x=\{y\in X: x<y\}$ has a minimum. For $x,y\in X$, we write $x\prec y$ if $x$ is covered by $y$, i.e. if there is no $z\in X$ such that $x < z < y$. If $x$ is a beat point (down or up), the inclusion $X-x\hookrightarrow X$ is a strong deformation retract and conversely, every strong deformation retract is obtained by removing beat points. A space without beat points is called a {\it minimal space}. Removing all beat points of $X$ leads to a minimal space called the \textit{core} of $X$. This core is unique up to homeomorphism, and two finite posets $X$ and $Y$ have the same homotopy type if and only if their cores are homeomorphic. Thus, a finite poset is contractible if and only if its core has a unique point. It is easy to see that a poset with maximum or minimum is contractible.

We recall now some basic facts about actions of finite groups. We refer the reader to \cite[Ch. 3]{Bre72} for more details about actions on simplicial complexes. All the actions considered here are on the right. If a group $G$ acts on a set $A$, then we denote the action of $g\in G$ on the element $a\in A$ by $a^g$. The orbit of an element $a\in A$ will be denoted by $\overline{a}$.

If $G$ is a group and $X$ is a $G$-poset, instead of removing a single beat point $x$, we can remove the orbit $\overline{x}$ and obtain an equivariant strong deformation retract $X-\overline{x}\hookrightarrow X$. It can be shown that if $f:X\to Y$ is an equivariant map which is also a homotopy equivalence, then $f$ is an equivariant homotopy equivalence \cite[Proposition 8.1.6]{Bar11a}. In particular, every $G$-invariant strong deformation retract is an equivariant strong deformation retract, $X$ has a $G$-invariant core, and if, in addition, it is contractible, $X$ has a fixed point.

A $G$-complex is a finite simplicial complex $K$ with an action of $G$ by simplicial automorphisms. Following the terminology of \cite{Bre72}, $K$ is said to satisfy property (A) if whenever $v,v^g$ are two vertices of $K$ in the same simplex, $v  = v^g$. $K$ is said to satisfy property (B) on $H\leq G$ if every time $\{v_0,\ldots,v_n\}$ and $\{v_0^{h_0},\ldots,v_n^{h_n}\}$ are simplices of $K$ with $h_i\in H$, there exists $h\in H$ such that $v_i^{h_i} = v_i^h$ for all $i$. $K$ satisfies property (B) if it satisfies property (B) on $G$, and it is regular if it satisfies property (B) on every $H\leq G$. Clearly property (B) implies property (A), and there are $G$-complexes not satisfying property (A) as the following example shows.

\begin{example}
Let $n$ be a positive integer. Let $K$ be the standard $n$-simplex and  $G$ be the cyclic group of order $n+1$. Then $G$ acts transitively on $K$ by permuting its vertices: if $G = \gen{g}$ and $\{v_0,\ldots,v_n\}$ are the vertices of $K$, define $v_i^g := v_{i+1}\mod n+1$ for $0\leq i\leq n$. Thus, $v_0,v_0^g = v_1$ are in the same simplex but $v_0\neq v_1$.
\end{example}

We denote by $K'$ the barycentric subdivision of the simplicial complex $K$.

\begin{proposition}[See \cite{Bre72}]
If $K$ is a $G$-complex, its barycentric subdivision $K'$ satisfies property (A). If $K$ satisfies property (A), $K'$ satisfies property (B). In particular, for every $G$-complex, its second barycentric subdivision satisfies property (B).
\end{proposition}

Note that by the previous proposition, the second barycentric subdivision $K''$ is always a regular $G$-complex.

For a finite poset $X$, its first subdivision is the poset $X'$ of non-empty chains of $X$. Equivalently, $X' = \X(\K(X))$, where for a finite simplicial complex $K$, $\X(K)$ denotes the poset of faces of $K$. Note that $K' = \K(\X(K))$.

\begin{proposition}
Let $X$ be a finite $G$-poset. Then $\K(X)$ is a $G$-complex and it satisfies property (A). In particular, $\K(X)' = \K(X')$ is a regular complex.
\end{proposition}

\begin{proof}
A simplex in $\K(X)$ is a non-empty chain in $X$ of the form $(x_0 < x_1 < \ldots < x_n)$. Thus, $G$ acts on $\K(X)$ by $(x_0 < x_1 < \ldots < x_n)^g = (x_0^g < x_1^g < \ldots < x_n^g)$. If $x\in \K(X)$ is a vertex and $x$, $x^g$ belong to a same simplex, they are comparable elements of $X$ and therefore equal by \cite[Lemma 8.1.1]{Bar11a}.
\end{proof}

From now on, we will make a distinction between a simplicial complex and its geometric realization. If $K$ is a simplicial complex, we denote by $|K|$ its geometric realization.

If $K$ is a $G$-complex, $|K|$ is a $G$-space. Hence, we can consider $|K|/G$, the orbit space. It has an induced cell structure which makes it a CW-complex. However, this structure may not be a triangulation for $|K|/G$ as the following example shows.

\begin{example}\label{exaFlipS1}
Let $X$ be the finite model of $\SS^1$ with four points. See Figure \ref{fig:Ex10}.

\begin{figure}[H]
    \begin{tikzpicture}
    \draw  (0.25,0.25)--(1.85,1.85);
    \draw  (0,0.25)--(0,1.75);
    \draw  (1.75,0.25)--(0.25,1.75);
    \draw (2,0.25)--(2,1.85);
    \node at (0,0) {$m_0$};
    \node at (0,2) {$M_0$};
    \node at (2,0) {$m_1$};
    \node at (2,2) {$M_1$};
    
    \draw (5.35,0.15)--(4.15,1.35);
    \draw (5.65,0.15)--(6.85,1.35);
    \draw (4.15,1.65)--(5.35,2.85);
    \draw (6.85,1.65)--(5.65,2.85);
    
    \node at (5.5,0) {$m_0$};
    \node at (4,1.5) {$M_0$};
    \node at (7,1.5) {$M_1$};
    \node at (5.5,3) {$m_1$};
	\end{tikzpicture}
\caption{Poset $X$ (left) and complex $\K(X)$ (right).}\label{fig:Ex10}
\end{figure}
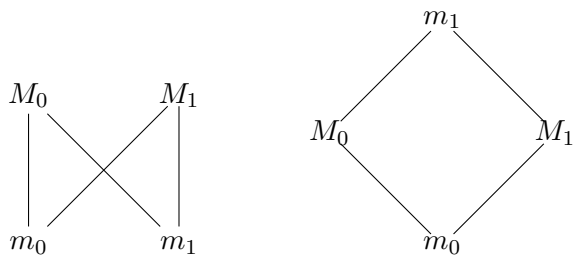
The cyclic group $\ZZ_2$ acts on $X$ by flipping the maximal elements and the minimal elements. The action induced on $|\K(X)|$ is the antipode action on $\SS^1$. The cellular structure induced on $|\K(X)|/\ZZ_2$ has two $0$-cells and two $1$-cells, and it is not a triangulation.

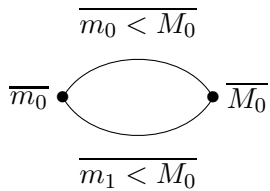
\begin{figure}[H]
\centering
\begin{tikzpicture}
    \draw (0,0) to[out=60,in=120] (2,0);
    \draw (0,0) to[out=-60,in=240] (2,0);
    
	\node at (0,0) {$\bullet$};
	\node at (2,0) {$\bullet$};
    
	\node at (1,1)  {$\overline{m_0 < M_0}$};
	\node at (1,-1)  {$\overline{m_1 < M_0}$};
    
    \node at (-0.45,0) {$\overline{m_0}$};
    \node at (2.45,0) {$\overline{M_0}$};
    
\end{tikzpicture}
\caption{Inherited cellular structure on $|\K(X)|/\ZZ_2$.}\label{fig:Ex12}
\end{figure}
\end{example}

If $K$ is a $G$-complex, the orbit complex $K/G$ is the simplicial complex whose vertices are the orbits of vertices of $K$, and the simplices are the sets of orbits of vertices $ \{\overline{v_0},\ldots,\overline{v_n}\}$ for which there exist representatives $w_i\in \overline{v_i}$ such that $\{w_0,\ldots,w_n\}$ is a simplex of $K$. In that case we say that $\{w_0,\ldots,w_n\}$ is a simplex above $\{\overline{v_0},\ldots,\overline{v_n}\}$. There is a simplicial map $p:K\to K/G$ which takes a vertex $v\in K$ to its orbit $\overline{v} \in K/G$. The following proposition says that for a regular complex $K$, this construction gives a triangulation for $|K|/G$ (see \cite{Bre72}).

\begin{proposition}
If $K$ is a regular $G$-complex, there is a homeomorphism $\varphi_K:|K|/G \to |K/G|$ induced by the quotient map $|K|\to |K|/G$.
\end{proposition}

In general, there is an induced map $\varphi_K:|K|/G\to |K/G|$ defined by
\[\varphi_K\left(\overline{\sum_i t_i v_i}\right) = \sum_i t_i \overline{v_i}.\]
It is just a continuous surjective map which may not be injective. 

If $X$ is a finite $G$-poset, the orbit space $X/G$ is a poset with the order $\overline{x}\leq \overline{x}'$ if there exist $y\in\overline{x}$ and $y'\in\overline{x}'$ such that $y\leq y'$. We have different orbit spaces arising from $X$, say $\K(X/G)$, $\K(X)/G$ and $|\K(X)|/G$, and we want to study the relationships between them.

\begin{example}
Let $X$ be the poset of Example \ref{exaFlipS1} with $G = \ZZ_2$. Then $X/G = \{\overline{m_0},\overline{M_0}\}$ and $\overline{m_0}<\overline{M_0}$. In particular it is contractible as finite space. The complex $\K(X)/G$ has two vertices $\overline{m_0}$, $\overline{M_0}$ and a single $1$-simplex $\{\overline{m_0},\overline{M_0}\}$. Consequently, $\K(X)/G$ is contractible. Since $|\K(X)|/G \equiv \SS^1$ is not contractible, in general $|K|/G$ and $|K/G|$ do not have the same homotopy type.
\end{example}

The following proposition is an immediate consequence of the definitions of the involved spaces.

\begin{proposition}
Let $X$ be a $G$-poset. Then $\K(X)/G$ is exactly the simplicial complex $\K(X/G)$.
\end{proposition}

It is easy to see that McCord's map is equivariant and it induces a continuous map on the quotient spaces $\hat{\mu}_X:|\K(X)|/G\to X/G$. We deduce the following proposition. 

\begin{proposition}\label{propCommutativeSquare}
If $X$ is a $G$-poset, we have a commutative diagram
\[\xymatrix{
|\K(X)|/G \ar[r]^{\hat{\mu}_X} \ar[d]^{\varphi_{\K(X)}} & X/G \\
|\K(X)/G|\ar@{=}[r] & |\K(X/G)| \ar[u]_{\mu_{X/G}}^\weak
}\]
where $\weak$ stands for weak equivalence. In particular, if $\varphi_{\K(X)}$ is an homeomorphism, $\hat{\mu}_X$ is a weak equivalence.
\end{proposition}

For a simplicial complex $K$, let $h:|K'|\to |K|$ be the homeomorphism defined by sending a simplex to its barycentre. If $K$ is a $G$-complex, also is $K'$ and $h$ is an equivariant map. In particular $\hat{h}:|K'|/G\to |K|/G$ is a homeomorphism.

The following commutative diagram shows the relationships between all the maps involved. Let $X$ be a $G$-poset and let $K = \K(X)$.

\begin{equation}\label{diagramaEspaciosDeOrbitas}
\xymatrix{
X' \ar[d] & |K'| \ar[l]_{\mu_{X'}}^{\weak} \ar[r]^h_{\equiv} \ar[d] & |K| \ar[d] \ar[r]^{\mu_X}_{\weak} & X \ar[d]\\
X'/G \ar@/_4pc/[dd]^{\alpha} & |K'|/G \ar[l]_{\hat{\mu}_{X'}}^{\weak} \ar[r]^{\hat{h}}_{\equiv} \ar[d]^{\varphi_{K'}}_{\equiv} & |K|/G \ar[r]^{\hat{\mu}_{X}} \ar[d]^{\varphi_K} & X/G\\
|\K(X'/G)| \ar[u]_{\mu_{X'/G}}^{\weak} \ar@{=}[r] & |K'/G| & |K/G| \ar@{=}[r] & |\K(X/G)| \ar[u]_{\mu_{X/G}}^{\weak}\\
(X/G)' & & |\K((X/G)')| \ar[ll]_{\mu_{(X/G)'}}^{\weak} \ar@{=}[r] & |\K(X/G)'| \ar[u]_h^{\equiv}
}
\end{equation}

Since $K = \K(X)$ satisfies (A), $K'$ is regular and $\varphi_{K'}:|K'|/G\to |K'/G|$ is a homeomorphism. In particular, $\hat{h}\circ \varphi_{K'}^{-1}:|K'/G|\to |K|/G$ gives a canonical triangulation for $|K|/G$.

In the diagram we have included the map $\alpha:X'/G\to (X/G)'$ defined by $$\alpha(\overline{(x_0 < x_1 < \ldots < x_n)}) = (\overline{x_0} < \overline{x_1} < \ldots < \overline{x_n}).$$ The following proposition shows that, in a certain way, $\alpha$ is the finite space version of the map $\varphi_K:|K|/G \to |K/G|$. We say that $X$ satisfies (B) if $\K(X)$ does.

\begin{proposition}
The map $\alpha$ is injective if and only if $X$ satisfies property (B). Moreover, if $\alpha$ is injective then it is an isomorphism of posets.
\end{proposition}

\begin{proof}
The first part is an easy restatement of property (B). To see the moreover part, take $c = (x_0 < \ldots < x_n)$ and $d = (y_0 < \ldots < y_m)$ to be two chains in $X'$ such that $\alpha(\overline{c})\leq \alpha(\overline{d})$. We have to see that $c^g \leq d$ for some $g\in G$. By injectivity we may assume that $n < m$. Since $(\overline{x_0} < \ldots < \overline{x_n}) \leq (\overline{y_0} < \ldots < \overline{y_m})$, there exists an injective order-preserving map $\tau:\{0,\ldots,n\}\to \{0,\ldots,m\}$ and elements $g_i \in G$ with $x_i^{g_i} = y_{\tau(i)}$. Thus, $y_{\tau(0)} < y_{\tau(1)} < \ldots < y_{\tau(n)} = x_0^{g_0}< x_1^{g_1} <\ldots < x_n^{g_n}$ is a chain of $X$. By property (B), there exists $g\in G$ such that $x_i^g = x_i^{g_i}$ for each $0\leq i\leq n$. Therefore, $$c^g = (x_0 < x_1 < \ldots < x_n) ^g = (x_0^{g_0} < x_1^{g_1} <\ldots < x_n^{g_n}) \leq (y_0 < y_1 < \ldots < y_m) = d.$$
\end{proof}

We denote by $X^{(n)}$ the $n$-th iterated subdivision of the poset $X$. We deduce the following corollaries.

\begin{corollary}\label{coroAllEquivalent}
If $X$ is a $G$-poset, for all $n\geq 1$ there is an isomorphism of posets $X^{(n)}/G \equiv (X'/G)^{(n-1)}$. If $X$ satisfies property (B), $X^{(n)}/G\equiv (X/G)^{(n)}$ for all $n\geq 0$.
\end{corollary}

\begin{proof}
Note that $X^{(n)}$ satisfies property (B) for $n\geq 1$. Assume $n\geq 2$. By the previous proposition, $X^{(n)}/G\equiv (X^{(n-1)}/G)'$, and by induction, $X^{(n-1)}/G\equiv (X'/G)^{(n-2)}$. Thus, $X^{(n)}/G\equiv ((X'/G)^{(n-2)})' \equiv (X'/G)^{(n-1)}$.

If $X$ satisfies property (B) then $X'/G\equiv (X/G)'$ and $X^{(n)}/G\equiv (X'/G)^{(n-1)}\equiv (X/G)^{(n)}$ for $n\geq 0$.
\end{proof}

\begin{corollary}\label{coroAllContractible}
For a $G$-poset $X$ and $n\geq 1$, $X^{(n)}/G$ is contractible if and only if $X'/G$ is contractible. If $X$ satisfies (B), $X^{(n)}/G$ is contractible if and only if $X/G$ is contractible.
\end{corollary}

\begin{proof}
By \cite[Corollary 4.18]{BM12}, a finite space $X$ is contractible if and only if its subdivision $X'$ is contractible. The result now follows from the corollary above and \cite[Corollary 4.18]{BM12}.
\end{proof}

We consider the action of $G$ by right conjugation on the posets of $p$-subgroups. That is, $A^g = g^{-1}Ag$ for $A\leq G$ and $g\in G$. The following example taken from \cite{Smi11} shows that $\S_p(G)$ may not satisfy property (B). Denote by $Z(G)$ the center of the group $G$ and by $N_G(A)$ the normalizer in $G$ of the subset $A\subseteq G$.

\begin{example}\label{exampleNoB}
Let $G = \SS_4$, the symmetric group on four letters, and let $X = \S_2(G)$. A Sylow $2$-subgroup of $G$ is $D= \gen{(1\, 3), (1\, 2\, 3\, 4)}\simeq D_8$. The elements $(1\, 3)(2\, 4)$ and $(1\, 2)(3\, 4)$ belong to $D$ and they are conjugated by $(2\, 3)\in G$. In this way, we have two different subgroups $H_1 = \gen{(1\, 3)(2\, 4)}$ and $H_2= \gen{(1\, 2)(3\, 4)}$ which determine the same point in $X/G$.

Take the chains $(H_1 < D)$ and $(H_2 < D)$. We affirm they have different orbits. Since $Z(D) = H_1$, if $(H_1 < D)^g = (H_2 < D)$, then $g\in N_G(D)\leq N_G(Z(D)) = N_G(H_1)$ and $H_2 = H_1^g = H_1$, a contradiction.
\end{example}

In the next section, we will show that $\S_p(G)'/G$ may not be contractible, although $\S_p(G)/G$ always is.

\section{Reformulation of Webb's conjecture and a stronger conjecture}

In \cite{Web87} P. Webb conjectured that $|\K(\S_p(G))|/G$ is contractible. Since the first proof of this conjecture due to P. Symonds (see  \cite{Sym98}), there have been various proofs and generalizations of this conjecture involving fusion systems and Morse Theory (see \cite{Bux99, Lib08, Lin09}). In all these articles, the authors work with the homotopy type of the orbit space $|K|/G$ for some simplicial complex $K$ $G$-homotopy equivalent to $\K(\S_p(G))$. For example, Symonds and Bux proved that $|\R_p(G)|/G$ is contractible. Here $\R_p(G)$ is the subcomplex of $\K(\S_p(G))$ consisting of chains $(P_0 <  P_1 <\ldots < P_n)$ with $P_i\normal P_n$ for all $i$. In \cite{TW91} it is proved that $\R_p(G)$ is $G$-homotopy equivalent to $\K(\S_p(G))$.

In this section we reformulate Webb's conjecture in terms of finite spaces by using the results of the previous section.

\begin{proposition}\label{propWebbConjReformulated}
(Webb's Conjecture) If $K\subseteq \K(\S_p(G))$ is a $G$-invariant subcomplex which is $G$-homotopy equivalent to $\K(\S_p(G))$, the finite space $\X(K)/G$ is homotopically trivial. In particular, it holds for $K \in\{ \K(\S_p(G)), \K(\A_p(G)), \K(\B_p(G)), \R_p(G)\}$.
\end{proposition}

Recall that $\B_p(G) = \{P\in \S_p(G) : P = \O_p(N_G(P))\}$,  where $\O_p(G)$ is the largest normal $p$-subgroup of $G$. The elements of $\B_p(G)$ are usually called \textit{$p$-radical subgroups} of $G$.

\begin{proof}
Since $K$ satisfies property (A), $K'$ is regular. By Proposition \ref{propCommutativeSquare}, the map $\hat{\mu}_{\X(K)}:|K'|/G \to \X(K)/G$ is a weak equivalence. Therefore, $\X(K)/G \approx_w |K'|/G \equiv |K|/G \simeq |\K(\S_p(G))|/G$. By the original Webb's conjecture, $|\K(\S_p(G))|/G$ is contractible and therefore, $\X(K)/G$ is homotopically trivial.
\end{proof}

In the context of finite spaces, being contractible is strictly stronger than being homotopically trivial. Hence, we could ask if $\X(K)/G$ is in fact contractible, if $K$ is one of the complexes $\K(\S_p(G))$, $\K(\A_p(G))$, $\K(\B_p(G))$ or $\R_p(G)$. The following examples show that it fails for $K = \K(\S_p(G))$ or $\K(\B_p(G))$.

\begin{example}\label{exampleSpG'/GnotContractible}
Let $G = \AA_6$ and $p = 2$. Then $\S_p(G)'/G$ is not a contractible finite space. The smallest example we have found for which $\S_p(G)'/G$ is not contractible is $G = \PSL(3,2)$, for $p = 2$.
\end{example}

In general, the poset $\B_p(G)'/G$ is not contractible but the example is much larger than that for $\S_p(G)$.

\begin{example}\label{exampleBpG'/GnotContractible}
Let $G$ be the transitive group of degree 26 and number 62 in the library of transitive groups of GAP. This group can be described as a semidirect product of a non-split extension of $\PSL(2,25)$ by $\ZZ_2$, by $\ZZ_2$, i.e. $G \simeq (\PSL(2,25).\ZZ_2)\rtimes \ZZ_2$. Here, the dot denotes non-split extension. Its order is $|G| = 2^5.3.5^2.13 = 31200$.

We have computed the poset $\B_p(G)'/G$ by using SageMath. Fix $P\in \Syl_p(G)$. Then $\B_p(G)/G = \{\overline{P},\overline{Q},\overline{R},\overline{A},\overline{B}\}$ with $A,B,R,Q\leq P$ $p$-radical subgroups of $G$ inside $P$. For particular $g,h\in G$, we have that $R\neq R^g$, $B\neq B^h$ and $R^g,B^h\leq P$. See Figure \ref{fig:Ex21} for the Hasse diagram of the poset. By calculating its core, it can be shown that the finite space $\B_p(G)'/G$ is not contractible.

\begin{figure}
\centering
\begin{tikzpicture}

    \draw (-0.25,0.25)--(-1.60,1.8);
    \draw (0.05,0.25)--(0,1.8);
    \draw (0.10,0.25)--(1.65,1.8);
    \draw (0.15,0.25)--(6.6,1.8);
    \draw (0.35,0.20)--(8.2,1.8);
    
	\draw (1.85,0.25)--(-1.3,1.8);
	\draw (2.15,0.25)--(10.5,1.8);

	\draw (3.85,0.25)--(0.40,1.8);
	\draw (3.95,0.25)--(3.25,1.8);
	\draw (4.15,0.25)--(5.1,1.8);
	
	\draw (5.85,0.25)--(1.85,1.8);
	
	\draw (7.75,0.15)--(3.35,1.8);
	\draw (7.80,0.20)--(5.4,1.8);
	\draw (7.85,0.25)--(7,1.8);
	\draw (8,0.25)--(8.5,1.8);
	\draw (8.15,0.25)--(10.8,1.8);
	
	\draw (-1.25,2.25)--(5.35,3.8);	
	
	\draw (0.15,2.25)--(-1,3.8);
	\draw (0.35,2.25)--(2.35,3.8);
	\draw (0.5,2.25)--(9.2,3.8);
	
	
	\draw (3.35,2.25)--(-0.75,3.8);
	\draw (3.55,2.25)--(10,3.8);
	
	\draw (5.3,2.25)--(2.55,3.8);
	
	\draw (6.8,2.25)--(-0.15,3.8);
	\draw (7,2.25)--(2.75,3.8);
		
	\draw (8.60,2.25)--(6,3.8);
	\draw (8.85,2.25)--(10.3,3.8);
	
	\draw (10.15,2.25)--(6.55,3.8);
    
	\node at (0,0) {$\overline{(P)}$};
	\node at (2,0) {$\overline{(Q)}$};
	\node at (4,0) {$\overline{(R)}$};
	\node at (6,0) {$\overline{(A)}$};
	\node at (8,0) {$\overline{(B)}$};
	
	\node at (-1.75,2) {$\overline{(Q<P)}$};
	\node at (0,2) {$\overline{(R<P)}$};
    \node at (1.75,2) {$\overline{(A<P)}$};
    \node at (3.5,2) {$\overline{(B<R^g)}$};
    \node at (5.25,2) {$\overline{(B<R)}$};
    \node at (7,2) {$\overline{(B<P)}$};
    \node at (8.75,2) {$\overline{(B^h<P)}$};
    \node at (10.5,2) {$\overline{(B^h<Q)}$};
   
    \node at (-1,4) {$\overline{(B<R^g<P)}$};
    \node at (2.5,4) {$\overline{(B<R<P)}$};
    \node at (6,4) {$\overline{(B^h<Q<P)}$};
    \node at (9.5,4) {$\overline{(B^h<R<P)}$};

\end{tikzpicture}
\caption{Finite space $\B_p(G)'/G$.}\label{fig:Ex21}
\end{figure}
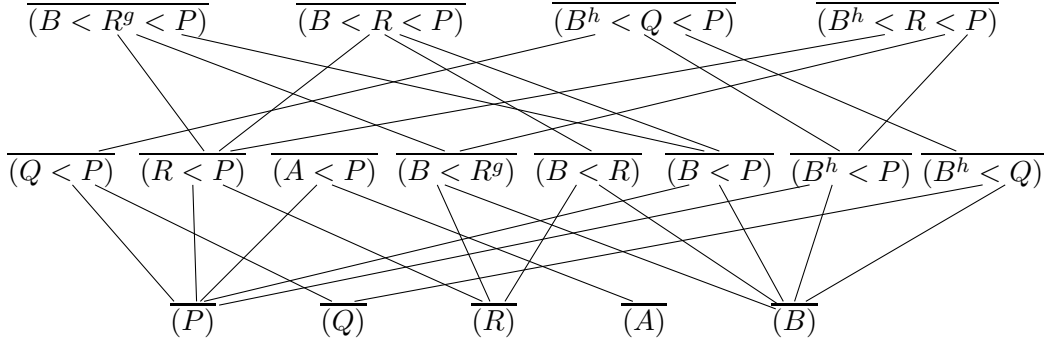
\end{example}

As we have said before, no example of a non-contractible poset $\A_p(G)'/G$ has been found. Recall that Webb's conjecture asserts that $\A_p(G)'/G$ is a homotopically trivial finite space. In Section 5 we prove some particular cases in which it is contractible by using both fusion and finite spaces. We believe that this stronger property holds in general. 

\begin{conjecture}\label{strongerConjecture}
The poset $\A_p(G)'/G$ is a contractible finite space.
\end{conjecture}

\begin{remark}
By Corollary \ref{coroAllEquivalent}, for a poset of $p$-subgroups $X$, the homotopy type (as finite space) of $X^{(n)}/G$ is determined by $X'/G$. Moreover, by Corollary \ref{coroAllContractible}, $X'/G$ is contractible if and only if some $X^{(n)}/G$ is contractible for $n\geq 1$.
\end{remark}

When trying to prove that $\A_p(G)'/G$ is contractible, the problem of how to control the fusion of chains of elementary abelian $p$-subgroups appears. This motivated us to introduce a new poset of $p$-subgroups in which it is easy to control the fusion of its chains.

Let $X_p(G)$ be the subposet of $\S_p(G)$ consisting of  the non-trivial $p$-subgroups normalized by all Sylow $p$-subgroup containing them. That is,
\[X_p(G) = \{Q\in \S_p(G) : Q\normal P \text{ for all }P\in \Syl_p(G) \text{ such that }Q\leq P\}.\]

In general, the subposet $X_p(G)\subseteq \S_p(G)$ is not weak equivalent to $\S_p(G)$. For instance, if $G = \SS_3\wr \ZZ_2$ and $p = 2$, $X_p(G)$ has the homotopy type of a discrete space of 9 points (as finite space), while $\S_p(G)$ is connected and has the weak homotopy type of a wedge of 16 $1$-spheres.

It is easy to see that $X_p(G)$ is invariant under the action of $G$, $\S_p(G)$ is contractible if and only if $X_p(G)$ is contractible, and $\S_p(G)$ is connected if $X_p(G)$ is. Note that the previous example shows that $\S_p(G)$ could be connected but $X_p(G)$ may not be.

We do not know if there is a stronger connection between $X_p(G)$ and $\S_p(G)$. Nevertheless, we believe they are weak equivalent when $X_p(G)$ is connected. This would suggest that Quillen's conjecture can be study in the poset $X_p(G)$: if $X_p(G)$ is homotopically trivial then it is contractible as finite space.

Now we show that $X_p(G)'/G$ is conically contractible. Fix $P\in \Syl_p(G)$. If $c$ is a chain in $X_p(G)'$ and all of its members are subgroups of $P$, then $c\cup(P)\in X_p(G)'$. We have the inequalities:
\[\overline{c}\leq \overline{c\cup (P)}\geq \overline{(P)}\]
which will turn into a homotopy of finite spaces after proving that the map $\overline{c}\mapsto \overline{c\cup (P)}$ does not depend on the representative $c\in \S_p(P)'\cap X_p(G)'$ chosen. Note that we are using Sylow's theorems.

Assume $c,c^g\in \S_p(P)'\cap X_p(G)'$ and $c = (Q_0< Q_1 < \ldots < Q_r)$. Then $P,P^g\leq N_G(Q_i^g)$ for all $i$ and $P,P^g$ are Sylow $p$-subgroups of $\bigcap_i N_G(Q_i^g)$. Take $h\in\bigcap_i N_G(Q_i^g)$ such that $P = P^{gh}$. Thus, $(c\cup (P))^{gh} = (c^{gh}\cup (P^{gh})) = (c^g\cup (P))$.

\begin{example}
A slightly modification of the previous ideas shows that $\N/G$ is contractible, where $\N\subseteq \S_p(G)'$ is the subposet of chains $c\in \S_p(G)'$ such that there exists a Sylow $p$-subgroup $P$ with $Q\normal P$ for all $Q\in c$. We can construct the homotopy in the same way we did in the previous example. First, fix a Sylow $p$-subgroup $P$. For $c\in \N$, take $g\in G$ with $Q^g\normal P$ for all $Q\in c$. We have a well-defined homotopy $\overline{c}\leq \overline{c^g\cup (P)}\geq \overline{(P)}$ in $\N/G$ (Cf. \cite[Theorem 3.2]{Lib08}).
\end{example}

\section{Contractibility of $\A_p(G)/G$}

Since the poset $G$ acts transitively on the maximal elements of $\S_p(G)$, the orbit space $\S_p(G)/G$ is contractible because it has a maximum (the orbit of any Sylow $p$-subgroup). The same holds for every $G$-invariant subposet $X\subseteq \S_p(G)$ which contains all the Sylow $p$-subgroups of $G$. For example, it holds for $ \B_p(G)$ or $X_p(G)$.

However, $G$ does not act transitively on maximal elements of $\A_p(G)$. In fact, the maximal elements of $\A_p(G)$ could have different orders. Hence, the homotopy type of $\A_p(G)/G$ cannot be deduced easily as in the case of $\S_p(G)$. Nevertheless, we show now that $\A_p(G)/G$ is contractible.

At this point, we will need a basic definition concerning the fusion of a group. We denote by $C_G(A)$ the centralizer in $G$ of the subset $A\subseteq G$.

\begin{definition}
Let $G$ be a finite group and fix $P\in\Syl_p(G)$. A subgroup $Q\leq P$ is said to be \textit{fully centralized} (fc for short) if $C_P(Q)\in \Syl_p(C_G(Q))$.
\end{definition}

\begin{remark}
It is easy to see that $Q\leq P$ is fully centralized if and only if $|C_P(Q^g)|\leq |C_P(Q)|$ for each $Q^g\leq P$.

Every subgroup $Q\leq P$ has a conjugated $Q^g$ by some $g\in G$ such that $Q^g\leq P$ is fully centralized: just take $Q^g\leq P$ maximizing $|C_P(Q^g)|$.
\end{remark}

\begin{theorem}\label{theoremApGGcontractil}
The poset $\A_p(G)/G$ is conically contractible.
\end{theorem}

\begin{proof}
Fix a Sylow $p$-subgroup $P\leq G$. Given that every orbit of $\A_p(G)/G$ can be represented by a fully centralized element inside $P$, we define the following homotopy: for $x\in \A_p(G)/G$, take $A\in x$ such that $A\leq P$ is fully centralized, and put
\[x = \overline{A}\leq \overline{\Omega_1(Z(\Omega_1(C_P(A))))}\geq \overline{\Omega_1(Z(P))}.\]
Recall that $\Omega_1(H)$ is the subgroup of $H$ generated by the elements of order $p$ in $H$. We are going to prove that the map $x = \overline{A} \mapsto \overline{\Omega_1(Z(\Omega_1(C_P(A))))}$ is well-defined, i.e. it does not depend on the choice of $A$, and it is order-preserving. After proving this, since $\Omega_1(Z(P))$ is always contained in the subgroups of the form $\Omega_1(Z(\Omega_1(C_P(A))))$, the result will follow.

\underline{Well-defined:} take $A,B\in x$ both fully centralized contained in $P$. We have to see that there exists $k\in G$ such that $\Omega_1(Z(\Omega_1(C_P(A))))^k = \Omega_1(Z(\Omega_1(C_P(B))))$.

Since $\overline{A} = \overline{B}$, $B = A^g$ for some $g\in G$. Note that $C_P(A)\in \Syl_p(C_G(A))$ implies $C_{P^g}(A^g)\in \Syl_p(C_G(A^g))$. Given that $C_P(A^g)\in \Syl_p(C_G(A^g))$, there exists $h\in C_G(A^g)$ such that $C_P(A^g) = C_{P^g}(A^g)^h = C_{P}(A)^{gh}$. Then, conjugation by $gh$ induces an isomorphism between $C_P(A)$ and $C_P(A^g)$. On the other hand, $\Omega_1(Z(\Omega_1(H)))$ is a characteristic subgroup of the group $H$. In conclusion, it must be $\Omega_1(Z(\Omega_1(C_P(A))))^{gh} = \Omega_1(Z(\Omega_1(C_P(A^g))))$.

\underline{Order-preserving:} suppose that $\overline{A} < \overline{B}$ with both $A,B\leq P$ fully centralized. We have to see that $$\overline{\Omega_1(Z(\Omega_1(C_P(A))))}\leq \overline{\Omega_1(Z(\Omega_1(C_P(B))))}.$$ Since $\overline{A} < \overline{B}$, there exists $g\in G$ such that $A< B^g$. However, it may happen that $B^g\nleq P$. We are going to fix this problem by using a trick that will be used repetitively along this article. Given that $C_P(A)$ is a Sylow $p$-subgroup of $C_G(A)$ and $B^g\leq C_G(A)$, there exists $h\in C_G(A)$ such that $B^{gh}\leq C_P(A)$. Moreover, we may choose $h$ in such a way that $B^{gh}\leq C_P(A)$ is fully centralized in $C_G(A)$ with the Sylow $p$-subgroup $C_P(A)$. This means that $C_{C_P(A)}(B^{gh})$ is a Sylow $p$-subgroup of $C_{C_G(A)}(B^{gh})$. But $C_{C_P(A)}(B^{gh}) = C_P(B^{gh})$ and $C_{C_G(A)}(B^{gh}) = C_G(B^{gh})$. Therefore, $A\leq B^{gh}$ and $B^{gh}\leq P$ is fully centralized in $G$.

By the well-definition, we obtain that $$\Omega_1(Z(\Omega_1(C_P(A))))\leq \Omega_1(Z(\Omega_1(C_P(B^{gh})))) = \Omega_1(Z(\Omega_1(C_P(B))))^k$$ for some $k\in G$.
\end{proof}

\begin{remark}
In the previous proof we have used the fact that if $A\leq B\leq P$ and $A$ is fully centralized, there exists $h\in C_G(A)$ such that $B^h\leq P$ is also fully centralized. This works because of the inclusion reversing of taking centralizers.

The other trick used here is that we always view the representative elements of an orbit inside a fixed Sylow $p$-subgroup. This will also be used repetitively.
\end{remark}

\begin{remark}
For $A\in \A_p(P)$, with $P$ a $p$-group, the subgroup $\Omega_1(Z(\Omega_1(C_P(A))))$ is just the intersection of all maximal elementary abelian $p$-subgroups of $P$ containing $A$.
\end{remark}

\begin{remark}
Take $X\in \{\A_p(G),\S_p(G),\B_p(G)\}$. If we knew that $X$ satisfies property (B), then $X'/G \equiv (X/G)'$ would be contractible by Corollary \ref{coroAllContractible}. However, it may not happen as we have shown in Example \ref{exampleNoB}.
\end{remark}

\section{Contractibility of $\A_p(G)'/G$}

In Section 3 we have seen that $\B_p(G)'/G$ and $\S_p(G)'/G$ are in general not contractible, and conjectured that $\A_p(G)'/G$ is always contractible. In this section we will prove some particular cases of this stronger conjecture.

In the previous section we have shown that $\A_p(G)/G$ is contractible by using a trick with the centralizers. This trick cannot be carried out in $\S_p(G)$ or $\B_p(G)$ because not every subgroup in these posets is abelian. The following proposition suggests that the property of being abelian makes things work.

\begin{proposition}\label{propMxlRank}
If $\A_p(G)\subseteq \S_p(G)$ is a strong deformation retract, both $\A_p(G)'/G$ and $\S_p(G)'/G$ are contractible finite spaces. In particular, this holds when the Sylow $p$-subgroups are abelian. Moreover, in the latter case $\B_p(G)\subseteq \S_p(G)$ is a strong deformation retract and $\B_p(G)'/G$ is a contractible finite space.
\end{proposition}

\begin{proof}
In \cite[Proposition 3.2]{MP18} it is shown that $\A_p(G)\subseteq \S_p(G)$ is a strong deformation retract if and only if $\Omega_1(P)$ is abelian for $P\in \Syl_p(G)$. In particular, this holds when the Sylow $p$-subgroups are abelian.

The hypothesis implies that $\A_p(G)\hookrightarrow\S_p(G)$ is an equivariant strong deformation retract. It induces an equivariant homotopy equivalence $\A_p(G)'\hookrightarrow \S_p(G)'$ and therefore, $\A_p(G)'/G$ and $\S_p(G)'/G$ have the same homotopy type as finite space. Now we show that $\A_p(G)'/G$ is contractible.

Fix a Sylow $p$-subgroup $P\leq G$ and always choose representatives of an element of $\A_p(G)'/G$ inside $\A_p(P)'$. By hypothesis, $\Omega_1(P)$ is the unique maximal element of $\A_p(P)$. If $\overline{(A_0 < A_1 < \ldots < A_r)}\in \A_p(G)'/G$, with $A_r\leq P$ fully centralized, then $(A_0 < A_1 < \ldots < A_r \leq \Omega_1(P))\in \A_p(P)'$. We are going to prove that the map
\[ c= \overline{(A_0 < A_1 < \ldots < A_r)}\in\A_p(G)'/G \mapsto f(c) = \overline{(A_0 < A_1 < \ldots < A_r \leq \Omega_1(P))}\]
is well-defined. Clearly it is order-preserving, and since $f(c) \geq \overline{(\Omega_1(P))}$, it will give us the contractibility of $\A_p(G)'/G$ after proving the well-definition.

Suppose that $\overline{(A_0 < A_1 < \ldots < A_r)} = \overline{(B_0 < B_1 < \ldots < B_r)}$ with $A_r,B_r\leq P$ both fully centralized. We have to prove that $$\overline{(A_0 < A_1 < \ldots < A_r \leq \Omega_1(P))} = \overline{(B_0 < B_1 < \ldots < B_r \leq \Omega_1(P))}.$$ There exists $g\in G$ such that $A_i^g = B_i$ for each $0\leq i \leq r$. Then
\begin{align*}
\overline{(B_0 < B_1 < \ldots < B_r \leq \Omega_1(P))} & = \overline{(A_0^g < A_1^g < \ldots < A_r^g \leq \Omega_1(P))}\\
& = \overline{(A_0 < A_1 < \ldots < A_r \leq \Omega_1(P)^{g^{-1}})}\\
& = \overline{(A_0 < A_1 < \ldots < A_r \leq \Omega_1(P)^{g^{-1}h})}
\end{align*}
where $h\in C_G(A_r)$ is such that $\Omega_1(P)^{g^{-1}h}\leq P$. Given that $\Omega_1(P)^{g^{-1}h}$ is generated by elements of order $p$, it is contained in $\Omega_1(P)$. Therefore, $\Omega_1(P)^{g^{-1}h} = \Omega_1(P)$ and $$\overline{(A_0 < A_1 < \ldots < A_r \leq \Omega_1(P))} = \overline{(B_0 < B_1 < \ldots < B_r \leq \Omega_1(P))}.$$

It remains to show that $\B_p(G)\subseteq \S_p(G)$ is a strong deformation retract when the Sylow $p$-subgroups are abelian. This is the content of the following proposition.
\end{proof}

\begin{proposition}
If the Sylow $p$-subgroups of $G$ are abelian, $\B_p(G)\subseteq \S_p(G)$ is a strong deformation retract. Moreover, since $\B_p(G)$ is invariant under the action of $G$, it is an equivariant strong deformation retract of $\S_p(G)$.
\end{proposition}

\begin{proof}
We prove that $\mathfrak{i}(\S_p(G)) = \B_p(G)$, where $\mathfrak{i}(\S_p(G))$ is the subposet of $\S_p(G)$ consisting of all non-trivial intersections of Sylow $p$-subgroups (see \cite{Bar11a,MP18}). Given that $\mathfrak{i}(\S_p(G))$ is always an equivariant strong deformation retract of $\S_p(G)$, the result will follow.

It is easy to show that in general every $p$-radical subgroup of $G$ is equal to the intersection of the Sylow $p$-subgroups containing it, so $\B_p(G)\subseteq \mathfrak{i}(\S_p(G))$. For the other inclusion take $Q\in \mathfrak{i}(\S_p(G))$ and note that $\Syl_p(N_G(Q)) = \{P\in\Syl_p(G): Q\leq P\}$ by abelianity. This implies that $$\O_p(N_G(Q)) = \bigcap_{P\in\Syl_p(G): Q\leq P} P = Q,$$ i.e. $Q\in\B_p(G)$.
\end{proof}

\begin{remark}\label{remarkXctr}
If $X$ is a contractible $G$-poset, $X'$ is contractible by \cite[Corollary 4.18]{BM12} and thus, its orbit space $X'/G$ is contractible by \cite[Proposition 8.3.14]{Bar11a}. In particular we have the following proposition.
\end{remark}

\begin{proposition}\label{propOpGcontractible}
If $\O_p(G) \neq 1$ the finite spaces $\S_p(G)'/G$ and $\B_p(G)'/G$ are contractible.
\end{proposition}

\begin{proof}
Since $\O_p(G)\neq 1$, $\S_p(G)$ and $\B_p(G)$ are contractible. In the former case, the homotopy is given by $P\leq P\O_p(G)\geq P$ and in the latter one, $\O_p(G)$ is the minimum of $\B_p(G)$. The result follows from the previous remark.
\end{proof}

\begin{remark}
From \cite{Sto84}, $\O_p(G)\neq 1$ if and only if $\S_p(G)$ is a contractible finite space. However, it does not imply that $\A_p(G)$ is contractible (see \cite{MP18}). Hence, the previous proposition does not work for $\A_p(G)'/G$.  Nevertheless, in most of the cases, $\A_p(G)$ is contractible when $\S_p(G)$ is, and we could expect $\A_p(G)'/G$ to be contractible when $\S_p(G)$ is.
\end{remark}

Similarly as in \cite{MP18}, we can prove the following result.

\begin{proposition}
If $|G| = p^\alpha.q$ for $q$ prime, $\A_p(G)'/G$, $\S_p(G)'/G$ and $\B_p(G)'/G$ are all contractible finite spaces.
\end{proposition}

\begin{proof}
By the analysis we have done in the proof of \cite[Proposition 3.2]{MP18}, we deduce that they are all contractible or the Sylow $p$-subgroups intersect trivially. The former case follows from Remark \ref{remarkXctr}. In the latter case, they are all $G$-homotopy equivalent since they have the $G$-homotopy type of a space with $|\Syl_p(G)|$ points in which $G$ acts transitively. The contractibility of $X'/G$ follows easily from this fact, for $X = \A_p(G)$, $\S_p(G)$ or $\B_p(G)$.
\end{proof}

Now we prove some particular cases for which $\A_p(G)'/G$ is contractible. These results will not hold in general for the posets $\S_p(G)'/G$ and $\B_p(G)'/G$ given that the techniques used here strongly use the fact that we are dealing with chains of elementary abelian $p$-subgroups.

Before that, we need a basic tool of fusion theory of groups, the Alperin's Fusion Theorem. We just require a weaker version of this theorem, which says that we can control the fusion inside a fixed Sylow $p$-subgroup via the normalizer of its non-trivial subgroups.

\begin{theorem}\label{theoremAlpFT}
Let $P\in\Syl_p(G)$ and suppose that $A,A^g\leq G$. Then there exist subgroups $Q_1,\ldots,Q_n\leq P$ and elements $g_i\in N_G(Q_i)$ such that:
\begin{enumerate}
\item $A^{g_1\ldots g_{i-1}}\leq Q_i$ for $1\leq i\leq n$,
\item $c_g|_{A} = c_{g_1\ldots g_n}|_A$.
\end{enumerate}
In particular, $A^g = A^{g_1\ldots g_n}$.
\end{theorem}

Here $c_g:H\to K$ denotes the conjugation map $c_g(h) = h^g$. A more general version of this theorem asserts that the $Q_i$'s can be taken to be essential subgroups of $P$ or even $P$. See \cite{AKO11, Cra11} for more details.

From now on, fix a Sylow $p$-subgroup $P\leq G$ and let $\Omega = \Omega_1(Z(P))$ be the subgroup generated by the central elements in $P$ of order $p$. We will use this subgroup to construct homotopies and extract beat points. Note that if $H\leq G$, $r_p(H)\leq r_p(G)$, and $r_p(G) = r_p(Q)$ if $Q$ is a Sylow $p$-subgroup of $G$.

We will deal with orbits of chains $\overline{c}\in \A_p(G)'/G$ and always assume that $c$ is chosen to be a representative of its orbit inside $\A_p(P)'$.

\begin{remark}
If $A\in \A_p(P)$ is a maximal elementary abelian $p$-subgroup of $P$, then $\Omega\leq A$. Thus, $\Omega$ is contained in the intersection of all maximal elementary abelian $p$-subgroups of $P$.
\end{remark}

We will prove some particular cases of Conjecture \ref{strongerConjecture} when the difference between the $p$-rank of $G$ and the $p$-rank of $Z(P)$ is small. Roughly, the difference $r_p(G) - r_p(\Omega)$, which only depends on the Sylow $p$-subgroup $P$ and its center, says how many non-central elements of order $p$ in $P$ there are. For example, if $r_p(G) - r_p(\Omega) = 0$ then $\Omega_1(P) = \Omega_1(Z(P))$, i.e. all elements of order $p$ in $P$ are central in $P$.

First we will prove the case  $r_p(G) - r_p(\Omega) \leq 1$. Later, we will prove some particular cases of the conjecture when $r_p(G) - r_p(\Omega) \leq 2$. In particular we will deduce that $\A_p(G)'/G$ is contractible when $|P|\leq p^4$. 

\begin{theorem}\label{theoremLowDifference}
If $r_p(G) - r_p(\Omega)\leq 1$, the finite space $\A_p(G)'/G$ is contractible.
\end{theorem}

\begin{proof}
The case $r_p(G) = r_p(\Omega)$ holds by Proposition \ref{propMxlRank} since $\A_p(G)\subseteq \S_p(G)$ is a strong deformation retract. Hence, we may assume that $r_p(G) - r_p(\Omega) = 1$. Consequently, if $A\in \A_p(P)$ does not contain $\Omega$, $A\Omega\in \A_p(P)$ is maximal.

Consider the set $\A = \{ A\in \A_p(P) : A \text{ is fc, }\Omega\nleq A \text{ and } \overline{A}\nleq \overline{\Omega}\}$. By the remark above, $\A$ does not contain maximal elements of $\A_p(P)$. 

Take representatives of conjugacy classes $A_1,\ldots,A_k\in \A$ such that if $A\in\A$ then $\overline{A} = \overline{A_i}$ for some $i$, and $\overline{A_i} < \overline{A_j}$ implies $j < i$. We will prove first that the subposet of orbits of chains $\overline{c}$ such that $\overline{A_i}\nleq \overline{c}$ for all $i$ is a strong deformation retract of $\A_p(G)'/G$. In order to do that, assuming we have extracted all orbits of chains containing $\overline{(A_j)}$ for $j < i$, we extract first all orbits containing $\overline{(A_i)}$ but not $\overline{(A_i\Omega)}$. Later we will extract those containing both $\overline{(A_i)}$ and $\overline{(A_i\Omega)}$.

Let $1 \leq i\leq k$ and let $A = A_i$. Assume we have extracted all possible orbits $\overline{c}$ such that $\overline{(A)}\leq \overline{c}$ but $\overline{(A\Omega)}\nleq \overline{c}$ as up beat points. If one of them still remains, take a maximal one, say $\overline{c}$.

Note that $\overline{c\cup (A\Omega)}$ was not extracted because $\overline{A\Omega}\neq \overline{A_j}$ for all $j$. We affirm that $\overline{c}$ is an up beat point covered by $\overline{c\cup (A\Omega)}$. Let $c = (B_1 < \ldots< B_s < A)$, where $s$ could be $0$, and let $d = c\cup (A\Omega) = (B_1 < \ldots < B_s < A < A\Omega)$. If $\overline{c} \prec \overline{d'}$, the representative $d'$ can be taken to have the form $d' = (B_1 < \ldots < B_s < A < (A\Omega)^g)$ by maximality of $\overline{c}$. Since $(A\Omega)^g\in \A_p(P)$ is maximal and contains both $A$ and $\Omega$, we have that $A\Omega = (A\Omega)^g$. Thus, $\overline{d} = \overline{d'}$ and $\overline{c}$ is an up beat point covered by $\overline{d}$.

At this point, we have extracted all elements containing $\overline{(A)}$ but not $\overline{(A\Omega)}$. We affirm that we can extract from bottom to top all elements containing both $\overline{(A)}$ and $\overline{(A\Omega)}$ because they are down beat points at the moment of their extraction. Take one of these, say $\overline{c}$ such that $\overline{A}\notin \overline{c'}$ for all $\overline{c'}<\overline{c}$ in the remaining subposet. Then $c = (B_1 < \ldots < B_s < A < A\Omega)$ and clearly $\overline{c}$ is a down beat point covering uniquely the element $\overline{(B_1 < \ldots < B_s < A\Omega)}$.

The remaining subposet is a strong deformation retract of $\A_p(G)'/G$ and has no orbit $x$ with $\overline{(A_j)}\leq x$ for $j\leq i$.

Therefore, we have a strong deformation retract of $\A_p(G)'/G$ whose elements have the following possibles representations:
\begin{equation}\label{eq1}
\overline{(C_1 < \ldots < C_s < \Omega < B)}
\end{equation}
\begin{equation}\label{eq2}
\overline{(C_1 < \ldots < C_s  < B)}
\end{equation}
\begin{equation}\label{eq3}
\overline{(C_1 < \ldots < C_s < \Omega)}
\end{equation}
\begin{equation}\label{eq4}
\overline{(C_1 < \ldots < C_s)}
\end{equation}
with $s\geq 0$ and $C_s < \Omega < B\leq P$. The aim is to prove that the map that includes $\Omega$ between the $C_s$'s and the $B$'s is well-defined and order-preserving. That is, if $\overline{(C_1 < \ldots < C_s)} = \overline{(C_1^g < \ldots < C_s^g)}$ then $\overline{(C_1 < \ldots < C_s < \Omega)} = \overline{(C_1^g < \ldots < C_s^g < \Omega)}$, and analogously with the elements of the form $\overline{(C_1 < \ldots < C_s  < B)}$.

Since $C_P(\Omega) = P$, $\Omega^g = \Omega$ if $\Omega^g\leq P$ is fully centralized. Thus, if $C_s, C_s^g\leq \Omega$ and $\overline{(C_1 < \ldots < C_s)} = \overline{(C_1^g < \ldots < C_s^g)}$, the subgroups $C_s,C_s^g$ are fully centralized and
\begin{align*}
\overline{(C_1 < \ldots < C_s < \Omega)} & = \overline{(C_1^g < \ldots < C_s^g < \Omega^g)}\\
& = \overline{(C_1^g < \ldots < C_s^g < \Omega^{gh})}\\
& =  \overline{(C_1^g < \ldots < C_s^g < \Omega)}.
\end{align*}
Here, $h\in C_G(C_s^g)$ is such that $\Omega^{gh}\leq C_P(C_s^g) = P$ is fully centralized.

For the other case, as before, assume that 
\[\overline{(C_1 < \ldots < C_s < B)} = \overline{(C_1^g < \ldots < C_s^g < B_{ }^g)}\]
with $C_s < \Omega < B\leq P$ and $C_s^g < \Omega < B^g\leq P$. By conjugating by $g^{-1}$ we obtain
\[\overline{(C_1^g < \ldots < C_s^g <\Omega < B^g)} = \overline{(C_1 < \ldots < C_s <\Omega^{g^{-1}} < B)}\]
Take $h\in C_G(C_s)$ such that $\Omega^{g^{-1}h}\leq C_P(C_s) = P$ is fully centralized. Therefore, $\Omega^{g^{-1}h} =\Omega$. Given that $B^{h}$ may not be a subgroup of $P$, we take a $h'\in C_G(\Omega)$ such that $B^{hh'}\leq P$. In particular, $h,h'\in C_G(C_s)$ and we have to prove that
\[\overline{(C_1 < \ldots < C_s <\Omega < B)} = \overline{(C_1 < \ldots < C_s <\Omega < B^{hh'})}\]
with $hh'\in C_G(C_s)$. We use Alperin's Fusion Theorem \ref{theoremAlpFT} inside the group $C_G(C_s)$ with $B,B^{hh'}\leq P\in \Syl_p(C_G(C_s))$. Hence, it remains to see the case $hh'\in N_{C_G(C_s)}(Q)$ where $Q\leq P$ and $B,B^{hh'}\leq Q$. Denote $k = hh'$. If $k\in N_G(\Omega)$ we are done, so assume $\Omega\neq \Omega^k$. Observe that $\Omega,\Omega^k\leq B^k$. Since $\Omega\leq Z(Q)$, $\Omega^k\leq Z(Q^k) = Z(Q)$ and $\Omega^k$ commute with all elements of order $p$ inside $Q$. In particular it commutes with $B$, and by maximality of this, $\Omega^k \leq B$. The condition $\Omega\neq \Omega^k$ implies $B = \Omega\Omega^k = B^k$ by an order argument. In any case, we have proved that $$\overline{(C_1 < \ldots < C_s < \Omega < B)} = \overline{(C_1 < \ldots < C_s < \Omega < B^k)}$$ as desired.

To complete the proof, note that the map that includes $\Omega$ inside the orbits of chains $\overline{c}$ represented as (\ref{eq1}), (\ref{eq2}), (\ref{eq3}), (\ref{eq4}) is an order-preserving map that satisfies
\[\overline{c} \leq \overline{c\cup (\Omega)} \geq \overline{(\Omega)}.\]
In consequence, $\A_p(G)'/G$ is contractible because it has a strong deformation retract which is.
\end{proof}

We obtain the following immediate corollary.

\begin{corollary}
If $\A_p(G)'/G$ has height $1$ then it is a contractible finite space.
\end{corollary}

By following the ideas of the theorem above, we focus now our attention in the case of $r_p(G)-r_p(\Omega) = 2$. We do not prove this case in general. Nevertheless, the techniques used here allow us to prove that $\A_p(G)'/G$ is contractible if $|G|_p\leq p^4$ or $|G|_p \leq p^{r_p(G)+1}$. Moreover, we find a strong deformation retract of $\A_p(G)'/G$ which can be used to prove it is contractible. We use again Alperin's Fusion Theorem \ref{theoremAlpFT}.

We begin with some preliminary lemmas.

\begin{lemma}\label{lemmaFCOmega}
Let $G$, $p$ and $P$ be as above. If $C\leq P$ is fully centralized, $C\Omega\leq P$ is fully centralized and $C_P(C\Omega) = C_P(C)$.
\end{lemma}

\begin{proof}
If $A,B\subseteq G$, then it holds $C_G(AB) = C_G(A)\cap C_G(B)$. In this way, $C_P(C\Omega) = C_P(C)\cap C_P(\Omega) = C_P(C)$ since $\Omega\leq Z(P)$. Let $g\in G$. Then $$|C_P((C\Omega)^g))| = |C_P(C^g)\cap C_P(\Omega^g)|\leq |C_P(C^g)|\leq |C_P(C)| = |C_P(C\Omega)|.$$
\end{proof}

\begin{remark}
If $r_p(G) - r_p(\Omega) = 2$, any maximal element of $\A_p(P)$ has $p$-rank $r_p(G)$ or $r_p(G)-1$. 
\end{remark}

\begin{lemma}\label{lemmaLowerThan2}
Let $G$, $p$ and $P$ be as above. Let $\P\subseteq \A_p(G)'/G$ be the following subposet.
\[\P = \{x\in \A_p(G)'/G:\text{if }x = \overline{c}\text{, with }c\in\A_p(P)' \text{, and }A\in c \text{ is fc, then } A\leq \Omega \text{ or } \Omega\leq A\}.\]
If $r_p(G) - r_p(\Omega) \leq 2$, $\P\subseteq \A_p(G)'/G$ is a strong deformation retract (as finite spaces).
\end{lemma}

\begin{proof}
If $r_p(G) - r_p(\Omega) \leq 1$, then the proof is similar to the one of Theorem \ref{theoremLowDifference}, so we may assume that $r_p(G) - r_p(\Omega) = 2$. Let $r = r_p(G)$ and let $$\A = \{ A\in \A_p(P) : A \text{ is fc, }\Omega\nleq A \text{ and } \overline{A}\nleq \overline{\Omega}\}.$$ Clearly, $\A\cap \Max(\A_p(P)) = \emptyset$ and $\Omega^g\notin \A$ for any $g\in G$.

Note that $\P$ is the subposet of orbits of chains not containing $\overline{(A)}$ for $A\in \A$. Therefore, we want to extract elements containing $\overline{(A)}$ with $A\in \A$, as beat points.

Assume inductively we have extracted all orbits of chains containing $\overline{(A)}$ for $A\in \A$ of $p$-rank at least $r'+1$, with $1\leq r' \leq r-1$. Take $A\in \A$ of $p$-rank $r'$. Then either $(A\Omega)^g\in \A$ for some $g\in G$, or else $(A\Omega)^g\notin \A$ for all $g\in G$.

\textit{Case 1:} there exists $g\in G$ with $(A\Omega)^g\in \A$. Thus, $(A\Omega)^g$ is fully centralized and does not contain $\Omega$. Since $r_p(\Omega) = r-2$, 
\begin{align*}
r_p((A\Omega)^g\Omega) & = r_p((A\Omega)^g) + r_p(\Omega) - r_p((A\Omega)^g\cap \Omega)\\
& \geq r_p((A\Omega)^g) + r_p(\Omega) - (r_p(\Omega)-1)\\
& = r_p(A\Omega) + 1 \\
& = r_p(A) + r_p(\Omega) - r_p(A\cap \Omega) + 1\\
& \geq r_p(A) + (r-2) - (r_p(A)-1) + 1\\
& = r
\end{align*}
That is, $(A\Omega)^g\Omega\in\A_p(P)$ is maximal of $p$-rank $r$. Given that $((A\Omega)^g\Omega)^{g^{-1}}\geq A$, there exists $x\in C_G(A)$ such that $((A\Omega)^g\Omega)^{g^{-1}x}\leq P$ is fully centralized. Therefore, we have the following inequalities
\begin{align*}
|C_P((A\Omega)^g\Omega)| & = |C_P((A\Omega)^g)| = |C_P(A\Omega)| = |C_P(A)|\\
& \geq |C_P(((A\Omega)^g\Omega)^{g^{-1}x})|\\
& \geq |C_P((A\Omega)^g\Omega)|
\end{align*}
which are in fact equalities. Moreover, we deduce that $C_P(A) = C_P(((A\Omega)^g\Omega)^{g^{-1}x})$. Let $D = ((A\Omega)^g\Omega)^{g^{-1}x}$. Observe that $D\in \A_p(P)$ is maximal because of its rank. Hence, $D =  \Omega_1(C_P(D)) = \Omega_1(C_P(A))$ and there is a unique maximal element above $A$.

Now take any orbit $\overline{c}$ containing $\overline{(A)}$ in the remaining subposet. If $\overline{(A<E)}\leq \overline{c}$, for some $E\leq P$ with $r_p(A) < r_p(E) \leq r-1$, by changing $E$ by a conjugate, we may assume that $E\leq P$ is fully centralized. Thus $\overline{E} \notin\overline{\A}$ and it implies $\Omega\leq E$, so that $E \geq A\Omega$. Because $r_p(A\Omega) \geq r-1$, we have the equality $E = A\Omega$, and it is a contradiction. Therefore no orbit of chain over $\overline{(A)}$ can contain an element of rank between $r_p(A)+1$ and $r-1$. Consequently, they have the form $\overline{(A_0 < \ldots < A_s < A)}$ or $\overline{(A_0 < \ldots < A_s <  A<E)}$ with $E\in \A_p(P)$ maximal of rank $r$. By the reasoning above, $E = D$ and $\overline{c}$ is an up beat covered by $\overline{(A_0 < \ldots < A_s < A < D)}$. Hence, we can remove all orbits containing $\overline{(A)}$ but not $\overline{(D)}$ from top to bottom since they are up beat points at the moment of their extraction.

After extracting all these elements, the elements containing $\overline{(A)}$ that remain have the form $\overline{(A_0 < \ldots < A_s < A<D)}$. Each one of them is a down beat point covering the element $\overline{(A_0 < \ldots < A_s < D)}$ if we extract them from bottom to top.

\textit{Case 2:} $(A\Omega)^g\notin \A$ for all $g\in G$. It is easy to see that $r_p(A\Omega) = r$ or $r-1$. In the former case, we can extract all elements containing $\overline{(A)}$ by using the same reasoning of the Case 1. In the latter case, we want to do something similar, but it may happen that $A$ has more than one maximal element of $\A_p(P)$ above it. If it has just one, it is similar to the proof of Case 1. Assume there are more than one maximal element above $A$. Note that they have $p$-rank $r$ because $r_p(A\Omega) = r-1$.

Like before, we extract first from top to bottom all orbits of chains containing $\overline{(A)}$ but not $\overline{(A\Omega)}$. These elements have the forms $\overline{(A_0 < \ldots < A_s < A)}$ and $\overline{(A_0 < \ldots < A_s < A < B)}$ with $B\in \A_p(P)$ maximal. As $\Omega \leq B$, $$\overline{(A_0 < \ldots < A_s <A<B)} < \overline{(A_0 < \ldots < A_s <A< A\Omega < B)}.$$ If $\overline{(A_0 < \ldots < A_s <A<B)} \prec \overline{d}$ at the moment of its extraction, then, after conjugating, $d$ can be taken to have the form $d = (A_0 < \ldots < A_s < A < A\Omega < B^g)$ for some $g\in C_G(A)$. We apply now Alperin's Fusion Theorem on $C_G(A)$ in order to prove that $$\overline{(A_0 < \ldots < A_s <A<A\Omega<B)} = \overline{(A_0 < \ldots < A_s <A<A\Omega<B^g)}$$ with the morphism $c_g:C_G(A)\to C_G(A)$. There exist subgroups $Q_1$, $\ldots$, $Q_r\leq C_P(A)$ and $g_i\in N_{C_G(A)}(Q_i)$ such that $B^{g_1\ldots g_{i-1}}\leq Q_i$ and $c_g|_B = c_{g_r}\circ\ldots\circ c_1|_B$. Let $B_i = B^{g_1\ldots g_i}$ and $B_0 = B$. Since they are all maximal and $A\leq B_i$, we have that $A\Omega\leq B_i$ for all $i$. Therefore it is enough to show that $$\overline{(A_0 < \ldots < A_s <A < A\Omega < B_i)} = \overline{(A_0 < \ldots < A_s <A < A\Omega < B_{i-1})}$$ for all $i \geq 1$. If $(A\Omega)^{g_i} = A\Omega$ we are done. Otherwise, note that $A\Omega\leq \Omega_1(Z(Q_i))$ and so $(A\Omega)^{g_i}\leq \Omega_1(Z(Q_i^{g_i})) = \Omega_1(Z(Q_i))$. It implies that $C = (A\Omega)(A\Omega)^{g_i} = \Omega_1(Z(Q_i))$ and in particular $\A_p(Q_i)$ has a unique maximal element which is $C$. Since $B_i, B_{i-1}\in \A_p(Q_i)$ are maximal elements, we deduce that $B_i = C = B_{i-1}$.

This has shown that $\overline{(A_0 < \ldots < A_s <A<B)}$ is an up beat point covered, at the moment of its extraction, by $\overline{(A_0 < \ldots < A_s <A<A\Omega < B)}$. After extracting it, the element $\overline{(A_0 < \ldots < A_s <A)}$ becomes an up beat point covered by $\overline{(A_0 < \ldots < A_s <A<A\Omega)}$ and we can extract it.

The remaining elements containing $\overline{(A)}$ can be extracted in the same way we did in Case 1.
\end{proof}

\begin{remark}
If $r_p(G)-r_p(\Omega) = 2$, all non-central and fully centralized elements $A\in \A_p(P)$ that could appear as element of a chain $c\in \A_p(P)'$ with $\overline{c}\in\P$ have $p$-rank $r_p(G)$ or $r_p(G)-1$.
\end{remark}

\begin{lemma}\label{lemmaUniqueMaximal}
With the notations of the lemma above, assume $r_p(\Omega) = r_p(G) - 2$. If $A\in \A_p(P)$ has rank $r_p(G)-1$ and it is covered by a unique maximal element in $\A_p(P)$, then $\P$ retracts by strong deformation to the subposet of elements not containing $\overline{(A)}$.
\end{lemma}

\begin{proof}
Clearly the result holds if such elements were extracted. Therefore, we may assume that every $A^g\leq P$  fully centralized contains $\Omega$. We can also take $A$ to be fully centralized. If $B = \Omega_1(C_P(A))$, $B\in \A_p(P)$ is the unique maximal element strictly containing $A$. We extract first the elements $x = \overline{(C_0 < \ldots < C_s < \Omega < A)}$, with $s\geq -1$, from top to bottom as up beat points. If $x \prec \overline{d}$, then $\overline{d} = \overline{(C_0 < \ldots < C_s <\Omega < A < D)}$ for some $D\in \A_p(P)$, and it has to be $D = B$ by uniqueness. Thus, $x$ is an up beat point. Moreover, any element of the form $\overline{(C_0^g < \ldots < C_s^g < \Omega^g < A)}$ is equal to $\overline{(C_0 < \ldots < C_s < \Omega < A^{h})}$ with $A^{h}\leq P$ fully centralized. It is easy to see that $A^h$ is also covered by a unique element, so $\overline{(C_0 < \ldots < C_s< \Omega < A^h)}$ is also an up beat point. After extracting all these orbits, the unique elements containing $\overline{(A)}$ are those of the form $\overline{(C_0 < \ldots < C_s< A<B)}$ and $\overline{(C_0 < \ldots < C_s< A)}$, for $s\geq -1$ and $C_s\leq \Omega$. We can extract $\overline{(C_0 < \ldots < C_s< A)}$ from top to bottom since they are up beats points covered by $\overline{(C_0 < \ldots < C_s< A<B)}$ at the moment of their extraction. Now the remaining elements containing $\overline{(A)}$ are $\overline{(C_0 < \ldots < C_s<A<B)}$, which we extract them from bottom to top since they are down beat points covering $\overline{(C_0 < \ldots < C_s<B)}$, and $\overline{(C_0^g < \ldots < C_s^g <\Omega^g < A <B)}$ which are down beat points covering $\overline{(C_0^g < \ldots < C_s^g <\Omega^g < B)}$.
\end{proof}

Now we prove a result which roughly says that $\A_p(G)'/G$ is a contractible finite space when the $p$-rank of $G$ and the rank of $\Omega$ are very close to $\log_p(|P|)$. In particular, if $P$ has order at most $p^4$ it will follow that $\A_p(G)'/G$ is contractible.

\begin{theorem}\label{theoremPDifferencesLower}
If $r_p(G) - r_p(\Omega)\leq 2$ and $r_p(G) \geq \log_p(|P|)-1$, then $\A_p(G)'/G$ is a contractible finite space.
\end{theorem}

\begin{proof}
By Proposition \ref{propMxlRank} and Theorem \ref{theoremLowDifference}, we may assume $r_p(G) = \log_p(|P|)-1$ and $r_p(G) - r_p(\Omega) = 2$. By Lemma \ref{lemmaLowerThan2}, we only need to show that $\P$ is contractible.

The idea is to extract beat points in order to reach a subposet with minimum $\overline{(\Omega)}$.

Let $A\in \A_p(P)$ be a non-maximal element of rank $r_p(G)-1$. There exists $B\in \A_p(P)$ of rank $r_p(G)$ such that $A < B$. It implies $B\leq C_P(A)$. On the other hand, $A\nleq Z(P)$ given that $r_p(A)  > r_p(\Omega) = r_p(Z(P))$, so $C_P(A) < P$. By order, $C_P(A) = B$ and $A$ is covered by a unique maximal element of $\A_p(P)$. In particular $A$ is fully centralized. By Lemma \ref{lemmaUniqueMaximal}, $\P$ retracts by strong deformation to the subposet of elements not containing $\overline{(A)}$. Hence, we get a strong deformation retract subposet of $\P$ whose elements are $\overline{(C_0 < \ldots < C_s < \Omega)}$, $\overline{(C_0 < \ldots < C_s < \Omega < B)}$ and $\overline{(C_0 < \ldots < C_s < B)}$, for $B\in \A_p(P)$ maximal of rank $r$ or $r-1$ and $s\geq -1$, and $\overline{(C_0 < \ldots < C_s)}$ for $s\geq 0$, with $C_s\leq\Omega$ in all cases.

Suppose that $B\in \A_p(P)$ is maximal of rank $r_p(G)-1$. It is easy to see by repeating the proof of Theorem \ref{theoremLowDifference} that all elements containing $\overline{(B)}$, which have the form $\overline{(C_0 < \ldots < C_s < B)}$ for $s\geq -1$ and $C_s\leq \Omega$, can be extracted from top to bottom as up beat points covered by $\overline{(C_0 < \ldots < C_s < \Omega < B)}$ at the moment of their extraction. Hence, any element containing $\overline{(B)}$ will also contain $\overline{(\Omega)}$.

We extract the elements containing $\overline{(B)}$ and not containing $\overline{(\Omega)}$ for $B\in \A_p(P)$ maximal of rank $r_p(G)$. These elements have the form $\overline{(C_0 < \ldots < C_s < B)}$ for $s\geq -1$ and $C_s\leq\Omega$, and are covered by $\overline{(C_0 < \ldots < C_s < \Omega < B)}$ and $\overline{(C_0 < \ldots < C_s < \Omega < B^g)}$ at the moment of their extraction, for some $g\in C_G(C_s)$. We need to prove they are equal. By using Alperin's Fusion Theorem, we can assume that $g\in N_{C_G(C_s)}(Q)$ for a subgroup $Q\leq P$ with $B,B^g\leq Q$. If $\Omega^g = \Omega$ we are done. Otherwise, $g\notin N_G(\Omega)$ and so $Q = B$ since it has to be $Q<P$ by order. This gives us $B = B^g$.

We have reached a strong deformation retract of $\P$ whose only elements not containing $\overline{(\Omega)}$ are $\overline{(C_0 < \ldots < C_s )}$ with $C_s\leq\Omega$. Again, by repeating the end of the proof of Theorem \ref{theoremLowDifference}, we can extract them from top to bottom as up beat points covered by $\overline{(C_0 < \ldots < C_s <\Omega)}$ at the moment of their extraction.

Finally, $\P$ retracts by strong deformation to a subposet with minimum $\overline{(\Omega)}$, and consequently it is contractible.
\end{proof}

As an easy consequence of Proposition \ref{propMxlRank} and Theorems \ref{theoremLowDifference} and \ref{theoremPDifferencesLower} we get the following corollary.

\begin{corollary}
If $|P|\leq p^4$, the finite space $\A_p(G)'/G$ is contractible.
\end{corollary}

\begin{remark}
Almost all the proofs that we have done can be carried out to a general saturated fusion system over a fixed $p$-group $P$. If $\F$ is a saturated fusion system over $P$, we can form the orbit poset $\A_p(P)/\F$ in the following way: if $A,B\in \A_p(P)$ define the relation $A\sim B$ if $\varphi(A) = B$ for some morphism $\varphi$ in the category $\F$. Then $\A_p(P)/\F :=\A_p(P)/\sim $ is contractible with the same homotopy that we have defined in Theorem \ref{theoremApGGcontractil}.

We can define the relation $\sim$ in $\A_p(P)'$ and set $\A_p(P)'/\F := \A_p(P)'/\sim$. Thus, in the case of $\F =\F_P(G)$, $\A_p(P)/\F = \A_p(G)/G$ and $\A_p(P)'/\F = \A_p(G)'/G$.
\end{remark}


\end{document}